\documentclass{article}
\usepackage[utf8]{inputenc}
\usepackage[T1]{fontenc}
\usepackage{lmodern}
\usepackage{microtype}
\usepackage[margin=1.35in]{geometry}
\usepackage{amssymb,amsmath,amsfonts,amsthm}
\usepackage{stmaryrd}
\usepackage[square,numbers]{natbib}
\usepackage[skip=4pt plus 1pt]{parskip}
\usepackage{authblk}
\usepackage[hidelinks]{hyperref}

\newtheorem{theorem}{Theorem}[section]
\newtheorem{lemma}[theorem]{Lemma}
\newtheorem{prop}[theorem]{Proposition}
\allowdisplaybreaks
\numberwithin{equation}{section}
\DeclareMathOperator{\gap}{gap}
\DeclareMathOperator{\Var}{Var}
\DeclareMathOperator{\sech}{sech}
\newcommand{\X}{\mathcal X}
\newcommand{\E}{\mathbb E}
\newcommand{\ind}{\mathbf 1}
\newcommand{\tv}{\mathrm{TV}}
\newcommand{\tmix}{t_{\mathrm{mix}}}
\newcommand{\pbar}{\overline{\pi_\beta}}
\newcommand{\Pbar}{\overline{P_\beta}}

\title{From torpid to rapid mixing: group averaging for a weakly interacting Ising star}
\author[1]{Michael C.H. Choi\thanks{Email: mchchoi@nus.edu.sg}}
\affil[1]{Department of Statistics and Data Science, National University of Singapore, Level 7, 6 Science Drive 2, 117546, Singapore}
\date{}

\begin{document}
\maketitle

\begin{abstract}
We study lazy single-site Metropolis dynamics $P_{\beta}$ at inverse temperature $\beta \geq 0$ on $\{-1,+1\}^d$ for an Ising star with additional signed interactions among the leaves. If the absolute row sums of the leaf-interaction matrix are at most $\kappa\le1/2$, the worst-case total-variation mixing time of $P_{\beta}$ is at least of order $d\exp\{c\beta(d-1)\}$ for $\beta\ge1$, with universal $c>0$. Averaging over global spin reversal reduces the mixing time to $\mathcal O(d^2(1+\beta))$ for both $GP_\beta G$ and $(P_\beta+G)/2$, where $G$ is the Gibbs kernel induced by the partition of the state space into spin-reversal orbits $\{-x,x\}$. Partition-function interpolation gives the lower bound for $P_\beta$. The upper bounds follow from Wu's Dobrushin inequality and a decomposition into projection and restriction chains. This gives an explicit example in which group averaging turns an exponentially slow chain into a polynomially fast one.
\end{abstract}

\section{Introduction and main results}\label{sec:intro}

We consider Ising stars with signed interactions among the leaves and show that averaging over global spin reversal reduces the worst-case total-variation mixing time from exponential to polynomial in the inverse temperature and dimension (i.e. the number of spins). The bounds hold uniformly over leaf interactions satisfying an absolute row-sum constraint.

Symmetry-based updates were studied through orbital Markov chains by \citet{Niepert}, and group-averaged kernels by \citet{ChoiWang, ChoiLimWang}. Projection chains are standard tools for analysing decomposable chains \cite{JSTV}. A related removal of a phase obstruction occurs for the Curie--Weiss model, where censored Glauber dynamics mixes rapidly at fixed inverse temperature above the critical value \cite{LLP,DLP}. For the perturbed Ising star, the projection representation and \citet[Theorem 2.1]{Wu} give bounds uniform in inverse temperature.

For integers $a\le b$, write $\llbracket a,b\rrbracket:=\{a,a+1,\ldots,b\}$. For a positive integer $d$, abbreviate $\llbracket d\rrbracket:=\llbracket1,d\rrbracket$.

Precisely, let $d\ge3$ and consider the state space
\[
\X=\{-1,+1\}^{d}.
\]
We write each configuration as $x=(x_0,x_1,\ldots,x_{d-1})$,
where $x_0$ is the spin at the centre of the star and
$x_i$, $i\in\llbracket d-1\rrbracket$, are the spins at its leaves. Let $K=(K_{ij})_{i,j\in\llbracket d-1\rrbracket}$ be symmetric, with $K_{ii}=0$, and suppose that
\begin{equation}\label{eq:kappa}
 \kappa:=\max_{i\in\llbracket d-1\rrbracket}\sum_{j\ne i}|K_{ij}|\le\frac12.
\end{equation}
Define the Hamiltonian and the target distribution by
\begin{equation}\label{eq:model}
 H(x)=-x_0\sum_{i=1}^{d-1} x_i-\sum_{1\le i<j\le d-1}K_{ij}x_ix_j,
 \qquad
 \pi_\beta(x)=\frac{e^{-\beta H(x)}}{\mathcal Z_\beta},
\end{equation}
where $\mathcal Z_\beta=\sum_{x\in\X}e^{-\beta H(x)}$ and $\beta\ge0$ is finite. Thus the interactions between the centre and the leaves have strength one. The graph of interactions among the leaves is arbitrary, subject to \eqref{eq:kappa}; these interactions may be ferromagnetic or antiferromagnetic.

For $i\in\llbracket0,d-1\rrbracket$, write $x^{(i)}$ for $x$ with spin $i$ reversed. The lazified single-site Metropolis kernel is
\begin{equation}\label{eq:P}
 P_\beta(x,y)=
 \begin{cases}
 \displaystyle\frac1{2d}\min\{1,e^{-\beta[H(y)-H(x)]}\},
       &y=x^{(i)}\text{ for some }i,\\[4pt]
 \displaystyle 1-\sum_{i=0}^{d-1}P_\beta(x,x^{(i)}),&y=x,\\
 0,&\text{otherwise}.
 \end{cases}
\end{equation}
It is irreducible, $\pi_\beta$-reversible, and satisfies $P_\beta(x,x)\ge1/2$, that is, $P_\beta$ is lazy.

Since $H(-x)=H(x)$, the global spin reversal $x\mapsto-x$ preserves $\pi_\beta$. Its orbits are the pairs $O_x=\{x,-x\}$. Let
\begin{equation}\label{eq:Gintro}
 R(x,y)=\ind_{\{y=-x\}},\qquad
 G=\frac12(I+R).
\end{equation}
The kernel $G$ draws uniformly from the current orbit and is independent of $\beta$. Both $GP_\beta G$ and $\frac12(P_\beta+G)$ preserve the original target distribution on $\X$ and are $\pi_{\beta}$-stationary.

For an ergodic kernel $T$ with stationary distribution $\mu$, define for $\varepsilon > 0$
\[
 \tmix(T,\varepsilon)
 :=\min\left\{t\in \mathbb{N} \cup \{0\}:
   \max_x\|T^t(x,\cdot)-\mu\|_{\tv}\le\varepsilon\right\}.
\]
Here $\|\eta-\mu\|_{\tv}:=\frac12\sum_x|\eta(x)-\mu(x)|$ is the total variation distance between $\eta$ and $\mu$. For reversible $T$, write $\gap(T)=1-\lambda_2(T)$, with eigenvalues ordered algebraically and $\lambda_2(T)$ being the second largest eigenvalue of $T$. Laziness of $P_\beta$ implies that it is positive semidefinite. Both averaged kernels $GP_\beta G$ and $\frac12(P_\beta+G)$ are also positive semidefinite, as verified in Lemma \ref{lem:averaging}. Thus all kernels used in our spectral mixing estimates have nonnegative eigenvalues.

\begin{theorem}\label{thm:main}
Under \eqref{eq:kappa}, the following statements hold for every $\beta\ge0$.
\begin{enumerate}
\item The original Metropolis dynamics satisfies
\begin{equation}\label{eq:slowgapmain}
 \gap(P_\beta)\le
 \frac1{d\cosh^{d-1}(\beta(1-\kappa))}.
\end{equation}
Consequently, for $0<\varepsilon<1/2$,
\begin{equation}\label{eq:slowmixmain}
 \tmix(P_\beta,\varepsilon)
 \ge\left[d\cosh^{d-1}(\beta(1-\kappa))-1\right]
       \log\frac1{2\varepsilon}.
\end{equation}
\item Let $\Pbar$ be the projection of $P_\beta$ onto the pairs $\{x,-x\}$ defined in \eqref{eq:generalprojection} and computed explicitly in Lemma \ref{lem:transitions} below. Then
\begin{align}
 \gap(GP_\beta G)
     &\ge\frac14\gap(\Pbar)
      \ge\frac{1-\kappa}{8d}\ge\frac1{16d},\label{eq:fastgapM}\\
 \gap(\tfrac12(P_\beta+G))
     &\ge\frac18\gap(\Pbar)
      \ge\frac{1-\kappa}{16d}\ge\frac1{32d}.\label{eq:fastgapA}
\end{align}
\item For $0<\varepsilon<1/2$,
\begin{align}
	\tmix(GP_\beta G,\varepsilon)
	&\le
	\left\lceil\frac{8d}{1-\kappa}
	\left(\log\frac1{2\varepsilon}
	+\frac{d\log2+\beta(2+\kappa)(d-1)}2\right)\right\rceil,
	\label{eq:fastmixM}\\
	\tmix\left(\frac12(P_\beta+G),\varepsilon\right)
	&\le
	\left\lceil\frac{16d}{1-\kappa}
	\left(\log\frac1{2\varepsilon}
	+\frac{d\log2+\beta(2+\kappa)(d-1)}2\right)\right\rceil.
	\label{eq:fastmixA}
\end{align}
\end{enumerate}
\end{theorem}

In particular, set $c_*:=\log\cosh(1/2)>0$. For every $\beta\ge1$, Theorem \ref{thm:main} implies
\begin{equation}\label{eq:separation}
\begin{aligned}
 \tmix(P_\beta,1/4)&\ge(d-1)(\log2)e^{c_*\beta(d-1)},\\
 \max\bigg\{\tmix(GP_\beta G,1/4),\tmix\left(\frac12(P_\beta+G),1/4\right)\bigg\}
 & =\mathcal O\bigl(d^2(1+\beta)\bigr),
\end{aligned}
\end{equation}
where the implicit constant is independent of $d$, $\beta$, and $K$ satisfying \eqref{eq:kappa}. The lower bound is exponential in $d$ at fixed $\beta\ge1$ and exponential in $\beta$ at fixed $d\ge3$.

At low temperatures, the leaves tend to align with the central spin.
Flipping the centre alone then raises the energy by an amount of
order $d$, so the Metropolis acceptance probability is exponentially
small in $\beta d$. This keeps the chain in one phase for a long time.
The kernel $G$ reverses all spins simultaneously with probability
$1/2$ without changing the Hamiltonian. Adding this update to form
$GP_\beta G$ or $(P_\beta+G)/2$ allows the chain to switch phases.
Rapid mixing of the projection chain then gives the polynomial
mixing time bounds.

Assume that a move of either $P_\beta$ or $G$ costs one unit
of work. Each step of $(P_\beta+G)/2$ chooses one of the two
kernels with equal probability and costs one unit, just as
a step of $P_\beta$ does. A direct implementation of
$GP_\beta G$ applies $G$, then $P_\beta$, and then $G$,
at a cost of three units per step. Under this convention,
Theorem~\ref{thm:main} gives an
$\mathcal O(d^2(1+\beta))$ upper bound on the worst-case
computational cost of either averaged kernel at
total-variation tolerance $1/4$. The corresponding cost
for $P_\beta$ is at least exponential in $\beta$ and $d$
for $\beta\ge1$.

The theorem concerns convergence of the specified dynamics and makes no claim of hardness for exact sampling.

Section \ref{sec:proofs} proves the theorem by analysing the original and projected spectral gaps and then deriving the mixing time bounds.

\section{Proofs of the main result}\label{sec:proofs}

\subsection{Projection and restriction chains; Gibbs kernel}\label{subsec:projection}

Following \citet[Section 2]{JSTV}, we recall the notions of projection and restriction chains. Let $O=(O_1,\ldots,O_k)$ be a partition of a finite state space into $k$ nonempty and pairwise disjoint blocks, and let $P$ be reversible with respect to a strictly positive distribution $\pi$. The projection chain on $\llbracket k\rrbracket$ is defined by
\begin{equation}\label{eq:generalprojection}
 \overline P(i,j):=
 \frac1{\pi(O_i)}\sum_{x\in O_i}\sum_{y\in O_j}\pi(x)P(x,y),
 \qquad \overline\pi(i):=\pi(O_i).
\end{equation}
It is reversible with stationary distribution $\overline\pi$. For $i\in\llbracket k\rrbracket$, the restriction chain on $O_i$ has transitions
\begin{equation*}
 P_{O_i}(x,y):=
 \begin{cases}
 P(x,y),&x\ne y,\\
 \displaystyle 1-\sum_{z\in O_i\setminus\{x\}}P(x,z),&x=y,
 \end{cases}
 \qquad x,y\in O_i.
\end{equation*}
Thus attempted exits from $O_i$ become holding probabilities. Detailed balance for $P$ shows that $P_{O_i}$ is reversible with stationary distribution $\pi_i(x):=\pi(x)/\pi(O_i)$, $x\in O_i$.

Writing $O_x$ for the block containing $x$, the Gibbs kernel for this partition is
\begin{equation}\label{eq:generalG}
 G_O(x,y)=
 \ind_{\{y\in O_x\}}\frac{\pi(y)}{\pi(O_x)}.
\end{equation}
It is the conditional expectation operator onto functions constant on each block, and hence is reversible and idempotent; see also \citet{ChoiWang}.

For the model \eqref{eq:model}, use the relative spins $z_i=x_0x_i$, $i\in\llbracket d-1\rrbracket$, so that $z=(z_1,\ldots,z_{d-1})\in\{-1,+1\}^{d-1}$, and write
\[
 O_z=\{(+1,z),(-1,-z)\}.
\]
The map $x\mapsto(x_0,z)$ is a bijection from $\X$ to $\{-1,+1\}\times\{-1,+1\}^{d-1}$. Define
\begin{equation*}
 S(z)=\sum_{i=1}^{d-1} z_i,\qquad
 V(z)=\sum_{i<j}K_{ij}z_iz_j.
\end{equation*}
Then $H(x)=-[S(z)+V(z)]$, and summation over the centre spin gives
\begin{equation}\label{eq:pbar}
 \mathcal Z_\beta=2\sum_{z\in\{-1,+1\}^{d-1}}e^{\beta[S(z)+V(z)]},\qquad
 \pbar(z)=\pi_\beta(O_z)=\frac{2e^{\beta[S(z)+V(z)]}}{\mathcal Z_\beta}.
\end{equation}
Each of the two states in $O_z$ has mass $\pbar(z)/2$, so \eqref{eq:generalG} is exactly the kernel $G$ in \eqref{eq:Gintro}. The restriction of $P_\beta$ to each $O_z$ is the identity kernel, since a single-spin move cannot send $x$ to $-x$ when $d\ge3$. The Gibbs kernel $G$ resamples from the conditional distribution on the orbit.

\begin{lemma}\label{lem:transitions}
For $z\in\{-1,+1\}^{d-1}$ and $i\in\llbracket d-1\rrbracket$, set
\[
 h_i(z_{-i})=1+\sum_{j\ne i}K_{ij}z_j.
\]
The off-diagonal transitions of the projection chain are
\begin{align}
 \Pbar(z,z^{(i)})
 &=\frac1{2d}\min\{1,e^{-2\beta z_i h_i(z_{-i})}\},\label{eq:projectleaf}\\
 \Pbar(z,-z)
 &=\frac1{2d}\min\{1,e^{-2\beta S(z)}\}.\label{eq:projectcentre}
\end{align}
All other off-diagonal entries vanish, and the diagonal completes each row to one.
\end{lemma}

\begin{proof}
Global reversal preserves the Hamiltonian and proposal probabilities, so $P_\beta(-x,-y)=P_\beta(x,y)$. The two representatives of $O_z$ therefore have the same transition probability into any specified orbit. In \eqref{eq:generalprojection}, the weighted average over these representatives equals their common probability.

Flipping leaf $i$ sends $z$ to $z^{(i)}$, and
\[
 S(z^{(i)})+V(z^{(i)})-S(z)-V(z)=-2z_i\left(1+\sum_{j\ne i}K_{ij}z_j\right).
\]
Flipping the centre sends $z$ to $-z$. Since $V(-z)=V(z)$,
$S(-z)+V(-z)-S(z)-V(z)=-2S(z)$. Substitution in \eqref{eq:P} gives the two formulas. The moves are distinct because $d\ge3$. In particular, the original dynamics is lumpable with respect to these orbits.
\end{proof}

For a reversible kernel $T$ with stationary distribution $\mu$, recall
\begin{equation}\label{eq:Dirichlet}
\begin{aligned}
 \mathcal E_T(f,f)
 &:=\langle f,(I-T)f\rangle_\mu
 =\frac12\sum_{x,y}\mu(x)T(x,y)(f(x)-f(y))^2,\\
 \gap(T)&=\inf_{\Var_\mu(f)>0}
       \frac{\mathcal E_T(f,f)}{\Var_\mu(f)}.
\end{aligned}
\end{equation}
Here $\langle f,g\rangle_\mu=\sum_x\mu(x)f(x)g(x)$ and
$\Var_\mu(f)=\mu(f^2)-\mu(f)^2$.

\subsection{An upper bound on the original spectral gap}\label{subsec:slow}

We estimate the stationary flow associated with changing the centre spin by interpolating the field in the projected distribution.

\begin{prop}\label{prop:slowgap}
The bound \eqref{eq:slowgapmain} holds. In fact, this bound only requires $\kappa<1$.
\end{prop}

\begin{proof}
For $a\in[0,1]$, define
\[
 \Xi_\beta(a)
 :=\sum_{z\in\{-1,+1\}^{d-1}}e^{\beta[aS(z)+V(z)]}.
\]
Thus $\Xi_\beta(1)=\mathcal Z_\beta/2$. Let $D=\{x:x_0=+1\}$. Global reversal gives $\pi_\beta(D)=1/2$, and the only single-site move from $D$ to $D^c$ flips the centre. Its stationary flow is
\begin{align*}
 \mathcal F_\beta(D,D^c)
 &:=\sum_{x\in D}\pi_\beta(x)P_\beta(x,D^c)\notag\\
 &=\frac1{4d\,\Xi_\beta(1)}
   \sum_z e^{\beta[S(z)+V(z)]}\min\{1,e^{-2\beta S(z)}\}\notag\\
 &=\frac1{4d\,\Xi_\beta(1)}
   \sum_z e^{\beta V(z)-\beta|S(z)|}
 \le\frac{\Xi_\beta(0)}{4d\,\Xi_\beta(1)}.
\end{align*}
Testing \eqref{eq:Dirichlet} with $\ind_D$, whose variance is $1/4$, yields
\begin{equation}\label{eq:gapratio}
 \gap(P_\beta)\le4\mathcal F_\beta(D,D^c)
 \le\frac{\Xi_\beta(0)}{d\,\Xi_\beta(1)}.
\end{equation}

Since $V$ is even and $S$ is odd under $z\mapsto-z$,
\begin{equation}\label{eq:tiltstart}
 \Xi_\beta(\kappa)
 =\sum_z e^{\beta V(z)}\cosh(\beta\kappa S(z))
 \ge\Xi_\beta(0).
\end{equation}
Write $\E_a$ for expectation under the distribution assigning
probability
\[
\frac{e^{\beta[aS(z)+V(z)]}}{\Xi_\beta(a)}
\]
to $z\in\{-1,+1\}^{d-1}$. Fix $a\in[\kappa,1]$ and
$i\in\llbracket d-1\rrbracket$. Conditional on all spins other
than $Z_i$, the terms in $aS(z)+V(z)$ involving $z_i$ are
\[
z_i\left(a+\sum_{j\ne i}K_{ij}z_j\right).
\]
The remaining terms do not depend on $z_i$ and therefore cancel
when normalising the conditional distribution. Consequently,
the conditional probabilities of $Z_i=+1$ and $Z_i=-1$ are
proportional to
\[
\exp\!\left(\beta\left[a+\sum_{j\ne i}K_{ij}z_j\right]\right)
\quad\text{and}\quad
\exp\!\left(-\beta\left[a+\sum_{j\ne i}K_{ij}z_j\right]\right),
\]
respectively. Taking their difference after normalisation and
using $\tanh u=(e^u-e^{-u})/(e^u+e^{-u})$ gives
\[
\E_a[Z_i\mid Z_{-i}]
=\tanh\!\left(\beta\left[a+\sum_{j\ne i}K_{ij}Z_j\right]\right).
\]
By \eqref{eq:kappa}, every configuration of the other spins satisfies
\[
a+\sum_{j\ne i}K_{ij}Z_j
\ge a-\sum_{j\ne i}|K_{ij}|
\ge a-\kappa.
\]
Since $\beta\ge0$ and $\tanh$ is increasing, we obtain
\[
\E_a[Z_i\mid Z_{-i}]
\ge\tanh(\beta(a-\kappa)).
\]
Taking expectations and applying the tower property yields
\[
\E_a[Z_i]
=\E_a\!\left[\E_a[Z_i\mid Z_{-i}]\right]
\ge\tanh(\beta(a-\kappa)).
\]
Summing over the $d-1$ coordinates therefore gives
\[
\E_a[S(Z)]
=\sum_{i=1}^{d-1}\E_a[Z_i]
\ge(d-1)\tanh(\beta(a-\kappa)).
\]

Finally, keeping $\beta$ fixed and differentiating the finite sum
defining $\Xi_\beta(a)$ term by term, we find
\begin{align*}
	\frac{d}{da}\log\Xi_\beta(a)
	&=\frac{\Xi_\beta'(a)}{\Xi_\beta(a)}\\
	&=\frac{\sum_{z\in\{-1,+1\}^{d-1}}\beta S(z)
		e^{\beta[aS(z)+V(z)]}}
	{\Xi_\beta(a)}\\
	&=\beta\E_a[S(Z)]\\
	&\ge(d-1)\beta\tanh(\beta(a-\kappa)).
\end{align*}
Integrating from $\kappa$ to $1$ and using \eqref{eq:tiltstart}, we obtain
\[
 \log\frac{\Xi_\beta(1)}{\Xi_\beta(0)}
 \ge\log\frac{\Xi_\beta(1)}{\Xi_\beta(\kappa)}
 \ge (d-1)\log\cosh(\beta(1-\kappa)).
\]
Together with \eqref{eq:gapratio}, this proves the proposition, including $\beta=0$.
\end{proof}

\subsection{A lower bound on the projected spectral gap}\label{subsec:fast}

The field in \eqref{eq:pbar} favours $z_i=+1$, and
\begin{equation}\label{eq:field}
 h_i(z_{-i})\ge1-\sum_{j\ne i}|K_{ij}|\ge1-\kappa>0
\end{equation}
for every boundary configuration. We use this uniform positive field to control the dependence between conditional distributions.

For $i\in\llbracket d-1\rrbracket$, let $B_i$ be the kernel which resamples coordinate $i$ from its conditional distribution under $\pbar$ and leaves all other coordinates unchanged. Its entries are
\begin{equation*}
 B_i(z,w)=\ind_{\{w_{-i}=z_{-i}\}}
 \frac{e^{\beta w_i h_i(z_{-i})}}
      {2\cosh(\beta h_i(z_{-i}))}.
\end{equation*}
Define the random-scan heat-bath kernel
\begin{equation*}
 Q_\beta=\frac1{d-1}\sum_{i=1}^{d-1}B_i.
\end{equation*}
Both $B_i$ and $Q_\beta$ are $\pbar$-reversible. For functions, $B_if=\E_{\pbar}[f\mid Z_{-i}]$. In particular, $B_i$ is an orthogonal projection in $\ell^2(\pbar)$ and $Q_\beta$ has nonnegative spectrum. Its off-diagonal entries are
\begin{equation}\label{eq:Qoff}
 Q_\beta(z,z^{(i)})=
 \frac1{(d-1)[1+e^{2\beta z_i h_i(z_{-i})}]},
 \qquad
 Q_\beta(z,w)=0\quad\text{if }w\ne z,z^{(1)},\ldots,z^{(d-1)}.
\end{equation}

\begin{lemma}\label{lem:Qgap}
Under \eqref{eq:kappa}, for every $\beta\ge0$,
\begin{equation}\label{eq:Qgap}
 \gap(Q_\beta)\ge\frac{1-\kappa}{d-1}\ge\frac1{2(d-1)}.
\end{equation}
\end{lemma}

\begin{proof}
Put
\[
 p_i(z_{-i}):=\pbar(Z_i=+1\mid Z_{-i}=z_{-i})
 =\frac{1+\tanh(\beta h_i(z_{-i}))}{2}.
\]
For distinct $i,j\in\llbracket d-1\rrbracket$, define the conditional influence
\begin{equation}\label{eq:cij}
 c_{ij}^{(\beta)}
 :=\sup_z\left|p_i(z_{-i})-p_i((z^{(j)})_{-i})\right|,
 \qquad c_{ii}^{(\beta)}=0.
\end{equation}
These are precisely the coefficients in \citet[equation (2.3)]{Wu} when the single-spin metric is $\delta(s,t)=\ind_{\{s\ne t\}}$. For this metric, the Wasserstein distance equals total-variation distance, and for two Bernoulli laws it is the absolute difference of their probabilities of $+1$.

The two local fields in \eqref{eq:cij} differ by $2K_{ij}z_j$ and both belong to $[1-\kappa,\infty)$. Since the derivative of
$u\mapsto(1+\tanh(\beta u))/2$ is $(\beta/2)\sech^2(\beta u)$, the mean-value theorem and \eqref{eq:field} imply
\begin{equation*}
 c_{ij}^{(\beta)}
 \le\beta|K_{ij}|\sech^2(\beta(1-\kappa)).
\end{equation*}
Let $C_\beta=(c_{ij}^{(\beta)})$ and
$\alpha_\beta=\max_i\sum_jc_{ij}^{(\beta)}$. For $u\ge0$,
$\sinh u=\int_0^u\cosh v\,dv\ge u$, so
\[
 \cosh^2u=1+\sinh^2u\ge1+u^2\ge2u.
\]
Consequently, for $\beta>0$,
\begin{equation}\label{eq:alpha}
 \rho(C_\beta)\le\alpha_\beta
 \le\beta\kappa\sech^2(\beta(1-\kappa))
 \le\frac{\kappa}{2(1-\kappa)}\le\kappa\le\frac12.
\end{equation}
For $\beta=0$, $C_\beta=0$, so the same conclusion holds.

By \citet[Theorem 2.1, equation (2.4)]{Wu},
\begin{equation}\label{eq:Wu}
 (1-\rho(C_\beta))\Var_{\pbar}(f)
 \le\sum_{i=1}^{d-1}\E_{\pbar}
       [\Var_{\pbar}(f\mid Z_{-i})].
\end{equation}
The moment assumptions in that theorem are automatic on this finite space. Since $B_i$ is conditional expectation,
\begin{align*}
 \E_{\pbar}[\Var_{\pbar}(f\mid Z_{-i})]
 &=\pbar(f^2)-\pbar((B_if)^2)\\
 &=\langle f,(I-B_i)f\rangle_{\pbar}.
\end{align*}
Summing over $i$, the right-hand side of \eqref{eq:Wu} is
$(d-1)\mathcal E_{Q_\beta}(f,f)$. Combining \eqref{eq:alpha}, \eqref{eq:Wu}, and the variational formula \eqref{eq:Dirichlet} proves \eqref{eq:Qgap}. Equivalently, Wu's continuous-time generator is $(d-1)(Q_\beta-I)$, which explains the factor $(d-1)$.
\end{proof}

\begin{prop}\label{prop:projectgap}
For every $\beta\ge0$,
\begin{equation*}
 \gap(\Pbar)\ge\frac{1-\kappa}{2d}\ge\frac1{4d}.
\end{equation*}
\end{prop}

\begin{proof}
For $w=z^{(i)}$, put $r=\pbar(w)/\pbar(z)$. Equations \eqref{eq:projectleaf} and \eqref{eq:Qoff} give
\[
 \Pbar(z,w)=\frac1{2d}\min\{1,r\},\qquad
 Q_\beta(z,w)=\frac1{d-1}\frac r{1+r}.
\]
Since $\min\{1,r\}\ge r/(1+r)$, we have
\begin{equation}\label{eq:comparison}
 \Pbar(z,w)\ge\frac{d-1}{2d}Q_\beta(z,w)\qquad(z\ne w).
\end{equation}
Note that, if $w$ is not a single-coordinate flip, the right-hand side vanishes. In particular, $Q_\beta(z,-z)=0$ because $d\ge3$, whereas \eqref{eq:projectcentre} supplies an additional move for $\Pbar$. Both chains have stationary distribution $\pbar$, so \eqref{eq:comparison} implies
\[
 \mathcal E_{\Pbar}(f,f)\ge\frac{d-1}{2d}\mathcal E_{Q_\beta}(f,f).
\]
The claim follows from Lemma \ref{lem:Qgap} and \eqref{eq:Dirichlet}.
\end{proof}

\subsection{Spectral gaps of the averaged kernels}\label{subsec:averaging}

\begin{lemma}\label{lem:averaging}
Both averaged kernels $GP_\beta G$ and $\tfrac12(P_\beta+G)$ are irreducible, $\pi_\beta$-reversible, positive semidefinite, and satisfy
\begin{equation*}
 \gap(GP_\beta G)\ge\frac14\gap(\Pbar),\qquad
 \gap(\tfrac12(P_\beta+G))\ge\frac18\gap(\Pbar).
\end{equation*}
\end{lemma}

\begin{proof}
Laziness makes $P_\beta$ positive semidefinite, and $G$ is an orthogonal projection in $\ell^2(\pi_\beta)$. Thus
\begin{align*}
 \langle f,GP_\beta Gf\rangle_{\pi_\beta}
 &=\langle Gf,P_\beta Gf\rangle_{\pi_\beta}\ge0,\\
 \langle f,\tfrac12(P_\beta+G)f\rangle_{\pi_\beta}
 &=\tfrac12\langle f,P_\beta f\rangle_{\pi_\beta}
   +\tfrac12\langle Gf,Gf\rangle_{\pi_\beta}\ge0.
\end{align*}
Both kernels are reversible. They are irreducible because, entrywise, $GP_\beta G\ge P_\beta/4$ and $\frac12(P_\beta+G)\ge P_\beta/2$.

For either averaged kernel $T$, its restriction to $O_x=\{x,-x\}$, in this order, is
\[
 T_{O_x}=\begin{pmatrix}1-a&a\\ b&1-b\end{pmatrix},\qquad
 a=T(x,-x),\quad b=T(-x,x).
\]
Its eigenvalues are $1$ and $1-a-b$, so $\gap(T_{O_x})=a+b$. Since $P_\beta(x,-x)=P_\beta(-x,x)=0$ for $d\ge3$ and $P_\beta(-x,-x)=P_\beta(x,x)$, we have
\begin{align*}
 (GP_\beta G)(x,-x)
 &=(GP_\beta G)(-x,x)\\
 &=\tfrac14\bigl(P_\beta(x,x)+P_\beta(-x,-x)\bigr)
 =\tfrac12P_\beta(x,x)\ge\tfrac14.
\end{align*}
For the additive kernel, both off-diagonal entries of the restriction are exactly $1/4$, because $G(x,-x)=G(-x,x)=1/2$. Consequently,
\begin{equation}\label{eq:restrictiongaps}
 \gap\bigl((GP_\beta G)_{O_x}\bigr)\ge\frac12,\qquad
 \gap\bigl((\tfrac12(P_\beta+G))_{O_x}\bigr)=\frac12.
\end{equation}
In both cases $1/4\le a=b\le1/2$, so these restrictions are irreducible and have nonnegative spectra.

By \citet[Proposition 2.5]{ChoiLimWang}, Gibbs averaging preserves the projection chain:
\begin{equation}\label{eq:averagedprojections}
 \overline{GP_\beta G}=\Pbar,\qquad
 \overline{\tfrac12(P_\beta+G)}=\tfrac12(\Pbar+I).
\end{equation}
The second identity follows directly from linearity of the projection construction and $\overline G=I$. Hence
\[
 \gap\bigl(\overline{GP_\beta G}\bigr)=\gap(\Pbar),\qquad
 \gap\bigl(\overline{\tfrac12(P_\beta+G)}\bigr)
 =\tfrac12\gap(\Pbar).
\]
The projection $\Pbar$ is lazy, since $P_\beta$ is lazy, so both projected kernels also have nonnegative spectra.

For either averaged kernel $T$, \citet[Theorem 3.2]{MartinRandall} gives
\begin{equation*}
 \gap(T)\ge\frac12\gap(\overline T)
       \min_{x\in\X}\gap(T_{O_x}).
\end{equation*}
The absolute spectral gap used in that theorem agrees with our definition of $\gap$ for the kernels considered here, since their spectra are nonnegative. Using \eqref{eq:restrictiongaps}--\eqref{eq:averagedprojections}, we obtain
\begin{align*}
 \gap(GP_\beta G)
 &\ge\frac12\gap(\Pbar)\cdot\frac12
  =\frac14\gap(\Pbar),\\
 \gap(\tfrac12(P_\beta+G))
 &\ge\frac12\cdot\frac12\gap(\Pbar)\cdot\frac12
  =\frac18\gap(\Pbar).
\end{align*}
This proves the lemma.
\end{proof}

\subsection{From spectral gaps to mixing times}\label{subsec:mixing}

We finish the proof by recalling the relevant spectral estimates; see also \citet[Chapter 12]{LPW}. For an ergodic reversible kernel $T$ with stationary distribution $\mu$, nonnegative spectrum, and gap $\gamma$, the $\ell^2(\mu)$ contraction and Cauchy--Schwarz give
\begin{equation*}
 \|T^t(x,\cdot)-\mu\|_{\tv}
 \le\frac12\sqrt{\mu(x)^{-1}-1}\,(1-\gamma)^t
 \le\frac1{2\sqrt{\mu(x)}}e^{-\gamma t}.
\end{equation*}
Writing $\mu_{\min}=\min_x\mu(x)$, it follows that
\begin{equation}\label{eq:mixuppergeneral}
 \tmix(T,\varepsilon)
 \le\left\lceil\gamma^{-1}
 \left(\log\frac1{2\varepsilon}+\frac12\log\frac1{\mu_{\min}}\right)
 \right\rceil.
\end{equation}

For the lower bound, suppose $0<\gamma<1$ and choose a nonconstant eigenfunction $f$ with $Tf=(1-\gamma)f$, $\mu(f)=0$, and $\|f\|_\infty=1$. At a state attaining $|f(x)|=1$, the dual characterisation of total variation yields
\[
 \|T^t(x,\cdot)-\mu\|_{\tv}\ge\frac12(1-\gamma)^t.
\]
Since $-\log(1-\gamma)\le\gamma/(1-\gamma)$,
\begin{equation}\label{eq:mixlowergeneral}
 \tmix(T,\varepsilon)
 \ge\frac{\log(1/(2\varepsilon))}{-\log(1-\gamma)}
 \ge(\gamma^{-1}-1)\log\frac1{2\varepsilon}.
\end{equation}

\begin{proof}[Proof of Theorem \ref{thm:main}]
Proposition \ref{prop:slowgap} proves \eqref{eq:slowgapmain}. The kernel $P_\beta$ is irreducible and lazy, and its gap is at most $1/d<1$. Applying \eqref{eq:mixlowergeneral} proves \eqref{eq:slowmixmain}. Proposition \ref{prop:projectgap} and Lemma \ref{lem:averaging} prove \eqref{eq:fastgapM}--\eqref{eq:fastgapA}.

To bound the minimum stationary mass, note that
\[
 \sum_{i<j}|K_{ij}|\le\frac{\kappa (d-1)}{2},\qquad
 |H(x)|\le(d-1)\left(1+\frac\kappa2\right).
\]
Therefore $\max_x H(x)-\min_x H(x)\le(2+\kappa)(d-1)$ and
\[
 \pi_{\beta,\min}\ge2^{-d}e^{-\beta(2+\kappa)(d-1)},\qquad
 \log\frac1{\pi_{\beta,\min}}
 \le d\log2+\beta(2+\kappa)(d-1).
\]
By Lemma \ref{lem:averaging}, both averaged kernels are irreducible and have nonnegative spectrum. Substituting the lower bounds on their spectral gaps and the
bound on $\pi_{\beta,\min}$ into \eqref{eq:mixuppergeneral} gives \eqref{eq:fastmixM}--\eqref{eq:fastmixA}.

Finally, for $a\ge0$, the function $b\mapsto\log\cosh(ab)$ is convex and vanishes at zero. Thus, for $\beta\ge1$,
\[
 \log\cosh(\beta(1-\kappa))
 \ge\beta\log\cosh(1-\kappa)
 \ge\beta\log\cosh(1/2)=\beta c_*.
\]
Since $dA-1\ge(d-1)A$ for $A\ge1$, \eqref{eq:slowmixmain} at $\varepsilon=1/4$ yields the first inequality in \eqref{eq:separation}. The remaining assertion follows from \eqref{eq:fastmixM}--\eqref{eq:fastmixA} and $1-\kappa\ge1/2$.
\end{proof}

\section*{Acknowledgements}

Michael Choi acknowledges financial support from the National University
of Singapore through projects A-0000178-02-00 and A-8003574-00-00.

During preparation of this manuscript, the authors used OpenAI's
ChatGPT (GPT-6.0 Astra) to assist with literature searches, checks of
mathematical arguments, language and \LaTeX{} editing. The author developed the ideas, manuscript, reviewed and edited all outputs, and take full responsibility for the
manuscript.
\begingroup
\small
\raggedright
\bibliographystyle{unsrtnat}
\bibliography{torpid-to-rapid_v4}

\begin{thebibliography}{9}
\providecommand{\natexlab}[1]{#1}
\providecommand{\url}[1]{\texttt{#1}}
\expandafter\ifx\csname urlstyle\endcsname\relax
  \providecommand{\doi}[1]{doi: #1}\else
  \providecommand{\doi}{doi: \begingroup \urlstyle{rm}\Url}\fi

\bibitem[Niepert(2012)]{Niepert}
Mathias Niepert.
\newblock {Markov} chains on orbits of permutation groups.
\newblock In \emph{Proceedings of the 28th Conference on Uncertainty in
  Artificial Intelligence}, UAI '12, pages 624--633, Catalina Island,
  California, USA, 2012. AUAI Press.
\newblock URL \url{https://arxiv.org/abs/1206.5396}.

\bibitem[Choi and Wang(2025)]{ChoiWang}
Michael C.~H. Choi and Youjia Wang.
\newblock Group-averaged {Markov} chains: mixing improvement, 2025.
\newblock URL \url{https://arxiv.org/abs/2509.02996}.
\newblock Preprint, arXiv:2509.02996.

\bibitem[Choi et~al.(2026)Choi, Lim, and Wang]{ChoiLimWang}
Michael C.~H. Choi, Ryan J.~Y. Lim, and Youjia Wang.
\newblock Group-averaged {Markov} chains {II}: tuning of group action in finite
  state space, 2026.
\newblock URL \url{https://arxiv.org/abs/2512.13067v2}.
\newblock Preprint, arXiv:2512.13067v2.

\bibitem[Jerrum et~al.(2004)Jerrum, Son, Tetali, and Vigoda]{JSTV}
Mark Jerrum, Jung-Bae Son, Prasad Tetali, and Eric Vigoda.
\newblock Elementary bounds on {Poincar\'e} and log-{Sobolev} constants for
  decomposable {Markov} chains.
\newblock \emph{The Annals of Applied Probability}, 14\penalty0 (4):\penalty0
  1741--1765, 2004.
\newblock \doi{10.1214/105051604000000639}.

\bibitem[Levin et~al.(2010)Levin, Luczak, and Peres]{LLP}
David~A. Levin, Malwina~J. Luczak, and Yuval Peres.
\newblock {Glauber} dynamics for the mean-field {Ising} model: cut-off,
  critical power law, and metastability.
\newblock \emph{Probability Theory and Related Fields}, 146:\penalty0 223--265,
  2010.
\newblock \doi{10.1007/s00440-008-0189-z}.

\bibitem[Ding et~al.(2009)Ding, Lubetzky, and Peres]{DLP}
Jian Ding, Eyal Lubetzky, and Yuval Peres.
\newblock Censored {Glauber} dynamics for the mean field {Ising} model.
\newblock \emph{Journal of Statistical Physics}, 137:\penalty0 407--458, 2009.
\newblock \doi{10.1007/s10955-009-9859-1}.

\bibitem[Wu(2006)]{Wu}
Liming Wu.
\newblock {Poincar\'e} and transportation inequalities for {Gibbs} measures
  under the {Dobrushin} uniqueness condition.
\newblock \emph{The Annals of Probability}, 34\penalty0 (5):\penalty0
  1960--1989, 2006.
\newblock \doi{10.1214/009117906000000368}.

\bibitem[Martin and Randall(2006)]{MartinRandall}
Russell Martin and Dana Randall.
\newblock Disjoint decomposition of {M}arkov chains and sampling circuits in
  {C}ayley graphs.
\newblock \emph{Combin. Probab. Comput.}, 15\penalty0 (3):\penalty0 411--448,
  2006.
\newblock ISSN 0963-5483,1469-2163.
\newblock \doi{10.1017/S0963548305007352}.
\newblock URL
  \url{https://doi-org.libproxy1.nus.edu.sg/10.1017/S0963548305007352}.

\bibitem[Levin and Peres(2017)]{LPW}
David~A. Levin and Yuval Peres.
\newblock \emph{{Markov} Chains and Mixing Times}.
\newblock American Mathematical Society, Providence, RI, second edition, 2017.
\newblock URL \url{https://pages.uoregon.edu/dlevin/MARKOV/}.
\newblock With contributions by Elizabeth L. Wilmer.

\end{thebibliography}
\endgroup
\end{document}